\documentclass[11pt, reqno]{amsart}

\usepackage{bm}
\usepackage{multirow}
\usepackage{tikz}
\usepackage{caption}
\usetikzlibrary{
 decorations.pathmorphing,%
 decorations.markings,%
 decorations.pathreplacing,%
 decorations.text,%
 calc,%
 patterns,%
 shapes.geometric,%
 arrows,
 decorations.shapes,
 positioning,
 plotmarks,
 shadings,
 math,
 intersections
}

\definecolor{jeffColor}{RGB}{102, 0, 204}
\definecolor{yaizaColor}{RGB}{0, 153, 153}
\definecolor{certainty}{RGB}{64, 228, 198}
\definecolor{hope}{RGB}{228, 194, 64}

\definecolor{periodColor}{RGB}{255, 167, 105}
\definecolor{dark-green}{RGB}{135, 194, 130}
\tikzset{>=latex} 
\tikzset{font=\small}
\tikzset{mark size=1.5pt, mark options=thin}
\tikzset{pin distance=4pt,
 every pin edge/.style={<-, thin, shorten <= -2pt}}
\usepackage{array,booktabs,tabularx}
\usepackage{graphicx}
\usepackage{amssymb}
\usepackage{enumerate}
\usepackage{color}
\usepackage{esint}
\usepackage{mathtools}
\usepackage{dsfont}
\usepackage{ esint }
\usepackage{placeins}
\usepackage[hide]{ed}
\usepackage{mathrsfs}
\usepackage{changepage}   
\usepackage{comment}
\usepackage{enumitem}
\usepackage{soul}
\usepackage{needspace}
\usepackage[normalem]{ulem}

\usepackage{multicol}
\usepackage{tabto}
\usepackage{hyperref}

\definecolor{uipoppy}{RGB}{221,128,71}
\definecolor{uipaleblue2}{RGB}{179,196,215}
\definecolor{uiviolet}{RGB}{86,86,99}
\definecolor{uiblack}{RGB}{0, 0, 0}
\definecolor{azul}{RGB}{0,128,255}
\definecolor{verde}{RGB}{50,180,50}
\definecolor{pale-verde}{RGB}{155,207,145}
\definecolor{uipaleblue}{RGB}{108,199,220}

\newtheorem{lemma}{Lemma}
\newtheorem{theorem}{Theorem}

\theoremstyle{definition}

\usepackage{dsfont}

\newcommand{\R}{{\mathbb R}}

\newcommand{\C}{{\mathbb C}}

\newcommand{\norm}[1]{\left\lVert#1\right\rVert}

\newcommand{\dell}{\ensuremath{\partial}}

\def\XXint#1#2#3{{\setbox0=\hbox{$#1{#2#3}{\int}$} \vcenter{\hbox{$#2#3$}}\kern-.5\wd0}}

\DeclareMathOperator{\supp}{supp}

\numberwithin{equation}{section}
\numberwithin{lemma}{section}

\newcommand*\Laplace{\mathop{}\!\mathbin\Delta}
\newcommand{\Lp}[2]{L^{#1}(#2)}
\newcommand{\Lpcomp}[2]{L^{#1}_{\text{comp}}(#2)}
\newcommand*\rfrac[2]{\frac{#1}{#2}}
\newcommand{\ang}[1]{\langle #1 \rangle}
\newcommand{\ind}[1]{\mathbf{1}_{#1}}
\newcommand*\Mellin[1]{\mathop{}\!\mathbin\mathcal{M}(#1)}
\renewcommand{\Im}{\operatorname{Im}}
\renewcommand{\Re}{\operatorname{Re}}
\newcommand{\Pconj}{P_{\varphi, E, \varepsilon}^{\pm}(h)}
\renewcommand{\S}{\mathbb{S}}
\newcommand{\N}{\mathbb{N}}
\newcommand{\Res}[2]{\operatorname{Res}_{#2}\left(#1\right)}
\newcommand*\Mellinverse[2]{\mathop{}\!\mathbin\mathcal{M}_{#1}^{-1}(#2)}
\newcommand{\TestFn}[1]{C_{c}^{\infty}(#1)}
\newcommand{\Ann}[2]{A(#1,#2)}

\title{Resolvent Bounds For Lipschitz Potentials In Dimension Two And Higher With Singularities At The Origin}
\author{Donnell Obovu}
\date{}
\address{Department of Mathematics, University College London, London, UK}
\email{}
\begin{document}
\maketitle
\begin{abstract}
We consider, for $h,E>0$, the semiclassical Schrödinger operator $-h^2\Delta + V - E$ in dimension two and higher. The potential $V$, and its radial derivative $\dell_{r}V$ are bounded away from the origin, have long-range decay and $V$ is bounded by $r^{-\delta}$ near the origin while $\dell_{r}V$ is bounded by $r^{-1-\delta}$, where $0\leq\delta\leq 4(\sqrt{2}-1)$. In this setting, we show that the resolvent bound is exponential in $h^{-1}$, while the exterior resolvent bound is linear in $h^{-1}$.
\end{abstract}

\section{Introduction}
Let $P$ denote the semiclassical Schrödinger operator on $\Lp{2}{\R^n}$, for $n\geq 2$, defined by
\begin{align}
\label{SchrodingerOperator}
P = P(h) := -h^2\Laplace + V : \Lp{2}{\R^n}\to\Lp{2}{\R^n},\qquad &h > 0,
\end{align}
where $\Laplace$ is the Laplacian on $\R^n$ and $V:\R^n\to\R$ is the potential. Unless otherwise stated, we will be working with polar coordinates $(r,\theta)=(|x|,\rfrac{x}{|x|})\in (0,\infty)\times\S^{n-1}$ to represent a point $x\in\R^n$. For a function $f$ defined on some subset of $\R^n$, we use the notation $f(r,\theta) := f(r\theta)$ and denote the partial derivative with respect to the radial variable by $f’ = \dell_{r}f$.

The potential $V$ must satisfy
\begin{gather}\label{VCond1}
V\in\Lp{p}{\R^n}+\Lp{\infty}{\R^n},\\\label{VCond2}
|V(r,\theta)|\ind{r<1}\leq c_1r^{-\delta},\\\label{VCond3}
|V(r,\theta)|\ind{r\geq 1}\leq y(r)\text{ for all } (r,\theta)\in(0,\infty)\times\S^{n-1}.
\end{gather}
Here, $p\geq 2$, $p>\frac{n}{2}$, $0\leq\delta <4(\sqrt{2}-1)$, $c_1>0$ and the function $y$ is non-negative, bounded, and decreases to zero as $r\to\infty$. We also require the distributional derivative of $V$ with respect to the radial coordinate $r$, $V'$ to exist and meet the criteria that
\begin{gather}\label{VCond4}
V'\in L^1_{\text{loc}}(\R^n\setminus\lbrace 0 \rbrace),\\\label{VCond5}
|V'(r,\theta)|\ind{0<r<1}\leq c_1r^{-1-\delta},\\\label{VCond6} |V'(r,\theta)|\ind{r>1}\leq c_0r^{-1}m(r)
\end{gather}
for some $c_0>0$ and a function $m:(0,\infty)\to [0,1]$ with
\begin{equation} \label{mCondition}
\lim_{r\to\infty}m(r)=0, \qquad (r+1)^{-1}m(r)\in\Lp{1}{0,\infty}.
\end{equation}
The typical examples of $m$, as shown in \cite{LongRangeLipschitz}, are $(1+r)^{-\rho}$ or $\log^{-1-\rho}(e+r)$ for $\rho> 0$.

The operator $P(h)$, with the potential $V$ satisfying \eqref{VCond1} to \eqref{VCond6}, is self-adjoint with respect to the domain $\mathcal{D}(P) = H^2(\R^n)$ \cite{FeynmannIntegrals}.

The main theorem of this paper is
\begin{theorem}\label{MainTheorem}
Fix $E>0$ and $s>\rfrac{1}{2}$. Suppose $V:\R^n\to\R$ satisfies conditions \eqref{VCond1} to \eqref{VCond6}. Then there exists $M = M(E,y,c_0,c_1,\delta,m), C_2 = C_2(E,y,c_0,c_1,\delta,m), C_3 = C_3(E,y,c_0,c_1,\delta,m) >0$ and $h_0\in(0,1]$ so that for all $\varepsilon > 0$ and $h\in (0,h_0]$,
\begin{gather}
\norm{\ang{x}^{-s}\left(P(h)-E\pm i\varepsilon\right)^{-1}\ang{x}^{-s}}_{\Lp{2}{\R^n}\to\Lp{2}{\R^n}}\leq e^{\rfrac{C_3}{h}},\label{GeneralFirstEstimate}
\intertext{and}
\norm{\ang{x}^{-s}\ind{|x|\geq M}\left(P(h)-E\pm i\varepsilon\right)^{-1}\ind{|x|\geq M}\ang{x}^{-s}}_{\Lp{2}{\R^n}\to\Lp{2}{\R^n}}\leq\frac{C_2}{h}.\label{GeneralSecondEstimate}
\end{gather}
where $\ang{x} := \ang{r} := (1+r^2)^{\frac{1}{2}}$.

Moreover, if $\supp V\subset B(0,R_0)$, then one can take $M = C_1(y,c_0,c_1,\delta,R_0)E^{-\rfrac{1}{2}}$.
\end{theorem}
The interest in resolvent estimates of the form \eqref{GeneralFirstEstimate} can be traced back to Burq \cite{Bu98}, who showed \eqref{GeneralFirstEstimate} in the case of smooth, compactly supported potentials. The resolvent estimates were then used to estimate the rate of decay of the local energy of the wave equation when an obstacle was present in the domain. Following works by Vodev, Burq and Cardoso and Vodev \cite{Vo00, Bu02, CV02} expanded on the resolvent estimates found in \cite{Bu98}, with \cite{Vo00} generalising the estimate to a class of noncompactly supported potentials, \cite{Bu02} extends the estimates in \cite{Bu98} to the case of smooth, long-range potentials, and Cardoso and Vodev \cite{CV02} refine Burq’s work \cite{Bu02} to give exterior estimates of the form seen in \eqref{GeneralSecondEstimate}.

Datchev’s work \cite{Da14} provided resolvent estimates \eqref{GeneralFirstEstimate} and \eqref{GeneralSecondEstimate} for potentials with weaker assumptions placed on their regularity in dimensions $n\neq 2$. The only requirement was that the potential $V$  and its radial derivative $\dell_{r}V$ be bounded and satisfy certain decay estimates. Datchev used an energy functional to prove global Carleman estimates, a technique which features in later works on resolvent estimates, including this paper. Further works that give resolvent estimates with little regularity assumed are \cite{Vo14, RoTa15, DadeH16, KlVo19, Vo19, Sh20, Vo20a, Vo20b, Vo20c, Vo20d, Vo22}.

Of particular importance to this paper is Shapiro's work \cite{Dim2}, which shows \eqref{GeneralFirstEstimate} and \eqref{GeneralSecondEstimate} for long-range potentials $V$ in two dimensions, expanding on \cite{Da14} to include the two dimensional case. However, \cite{Dim2} requires decay conditions on the whole gradient $\nabla V$ while \cite{Da14} only requires a derivative in the radial variable. This paper extends the results of \cite{Da14} to the two dimensional case, in particular, requiring only assumptions on $\dell_{r}V$, rather than $\nabla V$.

This paper is a continuation of the joint work of Galkowski and Shapiro \cite{LongRangeLipschitz}, who gave a proof for resolvent estimates in the case of potentials that are unbounded at the origin with long-range decay. We prove resolvent estimates for potentials with singularities at the origin in two dimensions and higher, extending the results of \cite{LongRangeLipschitz} to a class of potentials with greater growth at the origin. Notably, the resolvent estimates are true in the case of Coulomb potentials in three dimensions. We will see that, in order to handle singularities in a soon-to-be-defined energy functional at the origin, we require similar methods to those used by Galkowski and Shapiro in \cite{LongRangeLipschitz} with the added use of the Mellin transform inspired by \cite{CompactRPotentials}. In \cite{CompactRPotentials}, these Mellin transform methods were used to prove resolvent estimates in the case of compactly supported radial potentials that are only $L^\infty$. Here, we use these methods to handle large singularities as well as the case of dimension two simultaneously.


The approach taken to prove Theorem \ref{MainTheorem} involves defining the conjugated operator
\begin{align}\label{GeneralPConj}
\Pconj &:= e^{\rfrac{\varphi}{h}}r^{\rfrac{n-1}{2}}(P(h)-E\pm i\varepsilon )r^{-\rfrac{n-1}{2}}e^{-\rfrac{\varphi}{h}}\\\label{GeneralPConj2}
&\ =-h^2\dell_{r}^2+h^2r^{-2}\Lambda + 2h\varphi '\dell_{r} + V -(\varphi ')^2 + h\varphi '' - E\pm i\varepsilon,
\end{align}
where
\begin{gather}\label{GeneralLambdaOperator}
\Lambda := -\Laplace_{\S^{n-1}}+\frac{(n-1)(n-3)}{4},
\end{gather}
the phase function $\varphi$ is absolutely continuous and defined on $[0,\infty)$ with $\varphi\geq 0$, $\varphi (0) = 0$ and $\varphi ' \geq 0$ and $\Laplace_{\S^{n-1}}$ is the Laplace-Beltrami operator on $\S^{n-1}$. Throughout this paper, integrations are carried out with respect to the measure $drd\theta$ and the Lebesgue measure on $\R^n$. To avoid confusion, we will distinguish between $[0,\infty)\times\S^{n-1}$ and $\R^n$, with integration that takes place on subsets of $[0,\infty)\times\S^{n-1}$ being done with respect to $drd\theta$ and integration that takes place on $\R^n$ being done with respect to the Lebesgue measure. 

We now define an energy functional
\begin{gather}\label{EnergyFunctional}
F(r) = F[u](r) := \norm{u'(r,\cdot)}^2_{\Lp{2}{\S^{n-1}_{\theta}}} - \ang{(h^2r^{-2}\Lambda + V - (\varphi ')^2 - E)u(r,\cdot) , u(r,\cdot)}_{\Lp{2}{\S^{n-1}_{\theta}}},
\end{gather}
for $u\in r^{\frac{n-1}{2}}\TestFn{\R^n}$ when $n\geq 2$. For a weight function $w\in C^0[0,\infty)$ that is piecewise $C^1$, the distribution $(wF)'$ on $(0,\infty)$ is
\begin{multline}\label{wFderivative}
(wF)'= -2w\Re\ang{\Pconj u, u'}\mp2\varepsilon w\Im\ang{u,u'}+(2wr^{-1}-w')\ang{h^2r^{-2}\Lambda u,u} \\+ (4h^{-1}w\varphi '+w')\norm{hu'}^2 + (w(E+(\varphi ')^2-V))'\norm{u}^2 + 2w\Re\ang{h\varphi ''u,u'}, 
\end{multline}
where we drop the $\Lp{2}{\S^{n-1}_{\theta}}$ subscript for ease of notation. 
   
In the proof of Theorem \ref{MainTheorem} we attain a lower bound for $(wF)'$ by constructing appropriate weight and phase functions $w$ and $\varphi$. This requires extra work in dimension two compared to dimensions $n\neq 2$ because of the $(2wr^{-1}-w')\ang{h^2r^{-2}\Lambda u, u}$ term in \eqref{wFderivative}. In dimensions $n\neq 2$, the operator $\Lambda$ is non-negative, so this term can be bounded below by zero, provided we require $2wr^{-1}-w'\geq 0$. In dimension two, however, $\Lambda$ has a negative eigenvalue and the rest of the eigenvalues are positive. This means a negative singularity occurs at $r=0$, which requires extra care when compared to the cases when $n\neq 2$.

\subsection*{Acknowledgements} I would like to give thanks to my supervisor, Jeffrey Galkowski, who provided much appreciated guidance during the research of this topic and to Kiril Datchev and Jacob Shapiro, who read earlier versions of this paper and gave helpful remarks. I would also like to acknowledge support from the EPSRC Doctoral Training Partnership through grants EP/W523835/1 and EP/V001760/1.  

\section{Near Origin Estimates}
The first step is to establish control on the behaviour of $u$ near the origin. Much of this uses techniques used in Section 4 \cite{CompactRPotentials}, but adapted for the more general $n$ dimensional problem with potentials that need not be either $L^\infty$, compactly supported or radial.

We define an integral transform, $\mathcal{M}$, known as the Mellin transform, and its inverse $\mathcal{M}_t^{-1}$ by
\begin{align}\label{eq:Mellin0}
\Mellin{u}(\sigma, \theta) := \int_0^\infty r^{i\sigma}u(r,\theta)\frac{dr}{r},\qquad &\Mellinverse{t}{v}(r,\theta) = \frac{1}{2\pi}\int_{\Im\sigma = t}r^{-i\sigma}v(\sigma,\theta)d\sigma,
\end{align} 
for $u\in\TestFn{\R_+\times\S^{n-1}}$ and $v\in L^1_{\Re\sigma}(\R;L^2_{\theta}(\S^{n-1}))\cap L^2_{\Re\sigma}(\R;L^2_{\theta}(\S^{n-1}))$. By the density of $\TestFn{\R_+\times\S^{n-1}}$ and $L^1_{\Re\sigma}(\R;L^2_{\theta}(\S^{n-1}))\cap L^2_{\Re\sigma}(\R;L^2_{\theta}(\S^{n-1}))$ in $L^2(\R_+\times\S^{n-1}, r^{-2t-1}drd\theta)$ and $L^2_{\Re\sigma, \theta}(\R\times\S^{n-1})$ respectively, the Mellin transform and its inverse are well-defined, bounded operators $\mathcal{M}:\Lp{2}{\R_+\times\S^{n-1}; r^{-2t-1}drd\theta}\to L^2_{\Re\sigma, \theta}(\R\times\S^{n-1})$ and $\mathcal{M}_t^{-1}:L^2_{\Re\sigma, \theta}(\R\times\S^{n-1})\to\Lp{2}{\R_+\times\S^{n-1}; r^{-2t-1}drd\theta}$. The Mellin transform and its inverse satisfy the following for all $\theta\in\S^{n-1}$:
\begin{align}\label{eq:Mellin1}
\norm{\Mellin{u}(\sigma,\theta)}_{L^2_{\Re\sigma}(\R)} &= (2\pi)^{\rfrac{1}{2}}\norm{r^{-t-\rfrac{1}{2}}u(r,\theta)}_{L^2_{r}(\R_+)},\\\label{eq:Mellin2}
\norm{r^{-t-\rfrac{1}{2}}\Mellinverse{t}{v}(r,\theta)}_{L^2_{r}(\R_+)} &= (2\pi)^{-\rfrac{1}{2}}\norm{v(\sigma,\theta)}_{L^2_{\Re\sigma}(\R)},
\\\label{eq:Mellin3}
\Mellin{r^n\dell_{r}^nu}(\sigma,\theta) &= (-1)^n\frac{\Gamma(i\sigma + n)}{\Gamma(i\sigma)}\Mellin{u}(\sigma,\theta),
\end{align}
where $\sigma = \tau+it\in\C$ and $\Gamma (z)$ is the Gamma function \cite{debnath2016integral}.

We now define $\lambda_j := j^2 + (n-2)j + \rfrac{(n-1)(n-3)}{4}$ for $j\in\N_0$, which are the eigenvalues of $\Lambda$.
Let 
\begin{gather}
T := \left\lbrace t\in\R : t = \frac{1\pm\sqrt{1+4\lambda_j}}{2} \text{ for some } j\in \N_0 \right\rbrace\notag
\intertext{which is the set of the imaginary parts of the poles of the map $\sigma\mapsto (\sigma^2-i\sigma + \Lambda)^{-1}$. Elements of $T$ are of the form}\label{tform} t_{\pm, j} = \frac{1\pm \left((n-2)+2j\right)}{2}
\intertext{Additionally, define}
\Upsilon (t) := \norm{(t^2-t-\Lambda)^{-1}}_{\Lp{2}{\S^{n-1}}\to\Lp{2}{\S^{n-1}}}.\notag
\end{gather}
For the operator 
\begin{equation}\label{defn:Q}
Q := -\dell_{r}^2+r^{-2}\Lambda,
\end{equation}
we have the following result.

\begin{lemma}
\label{DecomposeU}
There exists a $C > 0$ such that, for $N\in\R$ with $-N-1\in \R\setminus T$, $t_0\in\R\setminus T$ and $u\in r^{-N}\Lpcomp{2}{\R_+\times\S^{n-1}}$ with $Qu\in r^{t_0-\rfrac{3}{2}}\Lpcomp{2}{\R_+\times\S^{n-1}}$,
\begin{gather}
u =E_{t_0}\left(r^2Qu\right)+\Pi_{t_0}\left(r^2Qu\right)\label{eq:DecomposeU}
\intertext{where for each $\theta\in\S^{n-1}$}\nonumber 
\norm{r^{-t_0-\rfrac{1}{2}}E_{t_0}(v)(r,\theta)}_{L^2_{r}(\R_+)}\leq C\Upsilon(t_0)\norm{r^{-t_0-\rfrac{1}{2}}v(r,\theta)}_{L^2_{r}(\R_+)}\\\intertext{and}\nonumber
\Pi_{t_0}(v)(r,\theta) := i\sum_{\substack{-N-1<\Im\sigma <t_0\\\sigma^2-i\sigma + \lambda_j = 0 \\\text{for some }j\in\N_0}}\Res{r^{-i\sigma}(\sigma^2-i\sigma+\Lambda)^{-1}\Mellin{v}(\sigma,\theta)}{\sigma}.
\end{gather}
\end{lemma}

\begin{proof}
Given $t_0\in\R\setminus T$, we can assume, without loss of generality, that $-N<t_0$.

Since $u\in r^{-N}\Lpcomp{2}{\R_+\times\S^{n-1}}$, the Mellin transform of $u$, $\Mellin{u}(\sigma,\theta)$, is holomorphic on the region $\Im\sigma<-N-\rfrac{1}{2}$.

Using \eqref{eq:Mellin3}, we can write 
\begin{equation*}
\Mellin{r^2Qu}(\sigma,\theta) = (\sigma^2-i\sigma+\Lambda)\Mellin{u}(\sigma,\theta),\qquad \Im\sigma < -N-\rfrac{1}{2},
\end{equation*}
which gives
\begin{align*}
u(r,\theta) &= \frac{1}{2\pi}\int_{\Im\sigma = -N-1}r^{-i\sigma}(\sigma^2-i\sigma + \Lambda)^{-1}\Mellin{r^2Qu}(\sigma,\theta)d\sigma\\
&=\frac{1}{2\pi}\lim_{R\to\infty}\int_{\gamma_{R, -N-1}}r^{-i\sigma}(\sigma^2-i\sigma + \Lambda)^{-1}\Mellin{r^2Qu}(\sigma,\theta)d\sigma,
\end{align*}
where $\gamma_{R,t} := \lbrace\sigma\in\C : \Re\sigma\in [-R,R], \Im\sigma = t\rbrace.$

We want to deform the contour to $\Im\sigma = t_0-\varepsilon$ for $\varepsilon > 0$ then take the limit as $\varepsilon\to 0$.

In the region $\Im\sigma < t_0$, $\norm{\Mellin{r^2Qu}(\sigma)}_{L_{\Re\sigma}^{\infty}}$ is finite, so we have
\begin{equation*}
\left|\int_{\gamma_{\pm R,-N,-\varepsilon}}r^{-i\sigma}(\sigma^2-i\sigma+\Lambda )^{-1}\Mellin{r^2Qu}(\sigma, \theta)d\sigma\right|\leq \frac{C_\varepsilon}{R^2}
\end{equation*}
where $\gamma_{\pm R,-N,-\varepsilon} := \lbrace\pm R + it : t\in [-N-1, t_0-\varepsilon]\rbrace$. In particular, using $t_0\in\R\setminus T$, $\Mellin{r^2Qu}(\sigma, \theta)$ varies continuously in $L^2_{\Re\sigma}(\R)$ for $\Im\sigma \leq t_0$, for each value of $\theta$. 

Sending $\varepsilon\to 0$ and $R\to\infty$ gives us
\begin{multline*}
u(r,\theta) = \frac{1}{2\pi}\int_{\Im\sigma = t_0}r^{-i\sigma}(\sigma^2-i\sigma+\Lambda)^{-1}\Mellin{r^2Qu}(\sigma,\theta)d\sigma \\+ i\sum_{\substack{-N-1<\Im\sigma <t_0\\\sigma^2-i\sigma + \lambda_j = 0 \\\text{for some }j\in\N_0}}\Res{r^{-i\sigma}(\sigma^2-i\sigma+\Lambda)^{-1}\Mellin{r^2Qu}(\sigma,\theta)}{\sigma}.
\end{multline*}
We can therefore define
\begin{align}
E_{t_0}(v)(r,\theta) &:= \frac{1}{2\pi}\int_{\Im\sigma = t_0}r^{-i\sigma}(\sigma^2-i\sigma+\Lambda)^{-1}\Mellin{v}(\sigma,\theta)d\sigma \notag\\\label{PiSum}
\Pi_{t_0}(v)(r,\theta) &:= i\sum_{\substack{-N-1<\Im\sigma <t_0\\\sigma^2-i\sigma + \lambda_j = 0 \\\text{for some }j\in\N_0}}\Res{r^{-i\sigma}(\sigma^2-i\sigma+\Lambda)^{-1}\Mellin{v}(\sigma,\theta)}{\sigma}.
\end{align}
Using \eqref{eq:Mellin0}, we can write 
\begin{equation*}
E_{t_0}(v) = \Mellinverse{t_0}{(\sigma^2-i\sigma+\Lambda)^{-1}\Mellin{v}(\sigma,\theta)}.
\end{equation*}
We now show that for $\Im{\sigma} = t_0$, $(\sigma^2-i\sigma+\Lambda)^{-1}$ on $\Lp{2}{\S^{n-1}}$ is bounded by $\Upsilon(t_0)$. As $\Lambda$ is self-adjoint, we have that
\begin{equation*}
\norm{(\sigma^2-i\sigma +\Lambda)^{-1}}_{\Lp{2}{\S^{n-1}}\to\Lp{2}{\S^{n-1}}} = \text{dist}(i\sigma-\sigma^2, \lbrace\lambda_j : j\in\N_0\rbrace)^{-1}.
\end{equation*}
To find an upper bound on $\text{dist}(i\sigma-\sigma^2, \lbrace\lambda_j : j\in\N_0\rbrace)^{-1}$, we notice that
\begin{equation*}
\sup_{\Im{\sigma} = t_0}(\text{dist}(i\sigma-\sigma^2, \lbrace\lambda_j : j\in\N_0\rbrace)^{-1})=\sup_{j\in\N_0}(\inf_{\Im{\sigma}=t_0}|i\sigma-\sigma^2-\lambda_j|)^{-1}.
\end{equation*}
Through calculation it can be showed that the minimum of $|i\sigma-\sigma^2-\lambda_j|$ with respect to $\Re{\sigma}$ occurs when $\Re{\sigma} = 0$, therefore
\begin{align*}
\sup_{\Im{\sigma} = t_0}(\text{dist}(i\sigma-\sigma^2, \lbrace\lambda_j : j\in\N_0\rbrace)^{-1})&= \sup_{j\in\N_0}|t_0^2-t_0-\lambda_j|^{-1}\\
&=\text{dist}(t_0^2-t_0, \lbrace\lambda_j : j\in\N_0 \rbrace)^{-1}\\
&=\norm{(t_0^2-t_0-\Lambda)^{-1}}_{\Lp{2}{\S^{n-1}}\to\Lp{2}{\S^{n-1}}}.
\end{align*}
Now, using \eqref{eq:Mellin1} and \eqref{eq:Mellin2} we obtain the desired result.
\end{proof}

Let $\alpha = \alpha (h) := \alpha_0h$ where $0<\alpha_0 < (2C\Upsilon (t_0))^{-\rfrac{1}{2}}$ and 
\begin{equation*}
a :=\min\left\lbrace\alpha, \left(\frac{h^2}{2Cc_1\Upsilon (t_0)}\right)^{\frac{1}{2-\delta}}\right\rbrace.
\end{equation*}
Here, $C$ is given by Lemma \ref{DecomposeU} and $c_1$ is from \eqref{VCond2}. Define $\alpha_1 := \max\lbrace\alpha,\rfrac{1}{2}\rbrace$. Additionally, for $R_1, R_2 \in [0,\infty]$, define $\Ann{R_1}{R_2} := \lbrace (r,\theta)\in (0,\infty) \times\S^{n-1} : R_1 < r < R_2\rbrace$ to be the annulus with inner and outer radii of $R_1$ and $R_2$ respectively. When integrating over $A(R_1,R_2)$, we will do so with respect to $drd\theta$. We have the following estimate for the behaviour of $u\in r^{\frac{n-1}{2}}\TestFn{\R^n}$ near the origin.

\begin{lemma}\label{lemNearOriginEstimate}
Suppose $E>0$, $t_0\in (-\rfrac{1}{2},\min(0,\rfrac{3}{2}-\delta))$ and the potential $V$ satisfies \eqref{VCond2}. Then there exists $C,h_0 >0$ such that for every $h\in (0,h_0]$, $0\leq\varepsilon\leq 1$ and $u\in r^{\frac{n-1}{2}}\TestFn{\R^n}$, we have

\begin{multline}\label{NearOriginEstimate}
\norm{r^{-\rfrac{1}{2}-t_0}u}_{\Lp{2}{\Ann{0}{\alpha_1}}} \leq C\Upsilon (t_0)h^{-2}\left(\norm{r^{\rfrac{3}{2}-t_0}r^{\rfrac{n-1}{2}}(P-E\pm i\varepsilon){r^{-\rfrac{n-1}{2}}}u}_{\Lp{2}{\Ann{0}{2\alpha_1}}}\right. \\
\left. \quad+ \norm{r^{\rfrac{3}{2}-t_0}(V-E\pm i\varepsilon )u}_{\Lp{2}{\Ann{a}{2\alpha_1}}}\right.\\
+ \left. h\norm{r^{\rfrac{3}{2}-t_0}hu'}_{\Lp{2}{\Ann{\alpha_1}{2\alpha_1}}} + h\norm{r^{\rfrac{3}{2}-t_0}hu}_{\Lp{2}{\Ann{\alpha_1}{2\alpha_1}}}\right)
\end{multline}

\begin{proof}
Let $\chi\in\TestFn{[0,2)}$ with $\chi\equiv 1$ in a neighbourhood of $[0,1]$. Define $\chi_{\alpha_1}(r):=\chi(\alpha_1^{-1}r)$. By \eqref{defn:Q} and \eqref{SchrodingerOperator}, we can write 
\begin{equation*}
Q = h^{-2}r^{\rfrac{n-1}{2}}(P-E\pm i\varepsilon)r^{-\rfrac{n-1}{2}}-h^{-2}(V-E\pm i\varepsilon)
\end{equation*}
therefore
\begin{multline}\label{Qchi} Q\chi_{\alpha_1}u = h^{-2}\chi_{\alpha_1}r^{\rfrac{n-1}{2}}(P-E\pm i\varepsilon)r^{-\rfrac{n-1}{2}}u + h^{-2}[r^{\rfrac{n-1}{2}}(P-E\pm i\varepsilon)r^{-\rfrac{n-1}{2}},\chi_{\alpha_1}]u \\- h^{-2}\chi_{\alpha_1}(V-E\pm i\varepsilon)u.
\end{multline}
Since $u\in r^{\frac{n-1}{2}}\TestFn{\R^n}$, $|V|\ind{r<1}\leq c_1r^{-\delta}$,  and $\chi_{\alpha_1}$ is constant near zero, $r^2Q\chi_{\alpha_1}u\in r^{2-\delta}\Lp{2}{\R_+\times\S^{n-1}}$, and, due to $t_0 < 0$ and Lemma \ref{DecomposeU}
\begin{equation}\label{abovee}
\chi_{\alpha_1}u = E_{t_0}(r^2Q\chi_{\alpha_1}u)+\Pi_{t_0}(r^2Q\chi_{\alpha_1}u),
\end{equation}
furthermore, we can multiply both sides of \eqref{abovee} by a function $\kappa$ such that $\kappa\equiv 1$ on the interval $[0,2]$ and $\kappa\equiv 0$ outside a compact set. Using Lemma \ref{DecomposeU} we get that $\kappa E_{t_0}(r^2Q\chi_{\alpha_1}u)\in\Lp{2}{\R_+\times\S^{n-1}}$, which implies $\kappa \Pi_{t_0}(r^2Q\chi_{\alpha_1}u)\in\Lp{2}{\R_+\times\S^{n-1}}$.

To calculate $\Pi_{t_0}(r^2Q\chi_{\alpha_1}u)$, we refer to \eqref{PiSum}. We need to find where the singularities of $r^{-i\sigma}(\sigma^2-i\sigma+\Lambda)^{-1}\Mellin{r^2Q\chi_{\alpha_1}u}$ occur in the $\sigma$ variable, which is done by solving for $\sigma_j$ in $\sigma_j^2-i\sigma_j+\lambda_j = 0$ for $-N-1<\Im{\sigma_j}<t_0$ for each $j\in\N_0$. 
In doing so we see that $\sigma_j = it_j$, where $t_j\in T$ can be written in the form given by \eqref{tform}. We also require $-N-1<\Im{\sigma_j}<t_0$, so each of the $\sigma_j$ have a negative imaginary part. Therefore, there are a finite amount of $\sigma_j$ being summed over, with the form $\sigma_j = i\frac{1}{2}(3-n-2j)$, where $j\geq 0$ for $n\geq 3$ and $j \geq 1$ when $n = 2$.
The $\Theta_j\in\Lp{2}{\S^{n-1}}$ are given by
\begin{equation*}
\Theta_j(\theta) = i\Res{r^{-i\sigma}(\sigma^2-i\sigma+\Lambda)^{-1}\Mellin{r^2Q\chi_{\alpha_1}u}(\sigma,\theta)}{\sigma = \sigma_j},
\end{equation*}
therefore we can write, for some $J\in\N$,
\begin{equation*}
\Pi_{t_0}(r^2Q\chi_{\alpha_1}u)(r,\theta) = \sum_{j=j_0}^Jr^{\rfrac{3-n}{2}-j}\Theta_j(\theta),
\end{equation*}
where $j_0 = 1$ when $n = 2$ or $n=3$ and $j_0 = 0$ otherwise.
 
For $\kappa\Pi_{t_0}(r^2Q\chi_{\alpha_1}u)$ to be in $\Lp{2}{\R_+\times\S^{n-1}}$, we must have $\Pi_{t_0}(r^2Q\chi_{\alpha_1}u) = 0$, in particular
\begin{equation*}
\chi_{\alpha_1}u = E_{t_0}(r^2Q\chi_{\alpha_1}u).
\end{equation*}
Using this fact, Lemma \ref{DecomposeU} and \eqref{Qchi}, we arrive at
\begin{align}
\notag\norm{r^{-\rfrac{1}{2}-t_0}u}_{\Lp{2}{\Ann{0}{\alpha}}}&\leq \norm{r^{-\rfrac{1}{2}-t_0}\chi_{\alpha_1}u}_{\Lp{2}{\R_+\times\S^{n-1}}}\\
\notag &= \norm{r^{-\rfrac{1}{2}-t_0}E_{t_0}(r^2Q\chi_{\alpha_1}u)}_{\Lp{2}{\R_+\times\S^{n-1}}}\\
\notag &\leq C\Upsilon (t_0)\norm{r^{\rfrac{3}{2}-t_0}Q\chi_{\alpha_1}u}_{\Lp{2}{\R_+\times\S^{n-1}}}\\\label{ineq0}
\begin{split}
&\leq C\Upsilon (t_0)h^{-2}\left(\norm{\chi_{\alpha_1}r^{\rfrac{3}{2}-t_0}r^{\rfrac{n-1}{2}}(P-E\pm i\varepsilon){r^{-\rfrac{n-1}{2}}}u}_{\Lp{2}{\R_+\times\S^{n-1}}}\right.\\
&\quad+\left.\norm{r^{\rfrac{3}{2}-t_0}[r^{\rfrac{n-1}{2}}(P-E\pm i\varepsilon){r^{-\rfrac{n-1}{2}}},\chi_{\alpha_1}]u}_{\Lp{2}{\R_+\times\S^{n-1}}}\right. \\&\quad+\left.\norm{\chi_{\alpha_1}r^{\rfrac{3}{2}-t_0}(V-E\pm i\varepsilon )u}_{\Lp{2}{\R_+\times\S^{n-1}}}\right).
\end{split}
\end{align}
We have defined $\chi_{\alpha_1}$ so that it is supported on the interval $[0,2\alpha_1]$ and bounded above by $1$, therefore
\begin{equation}\label{ineq1}
\norm{\chi_{\alpha_1}r^{\rfrac{3}{2}-t_0}r^{\rfrac{n-1}{2}}(P-E\pm i\varepsilon){r^{-\rfrac{n-1}{2}}}u}_{\Lp{2}{\R_+\times\S^{n-1}}}\leq \norm{r^{\rfrac{3}{2}-t_0}r^{\rfrac{n-1}{2}}(P-E\pm i\varepsilon){r^{-\rfrac{n-1}{2}}}u}_{\Lp{2}{\Ann{0}{2\alpha_1}}}.
\end{equation}
A straightforward calculation of the term involving the commutator gives, 
\begin{equation*}
[r^{\rfrac{n-1}{2}}(P-E\pm i\varepsilon){r^{-\rfrac{n-1}{2}}},\chi_{\alpha_1}]u=-h^2(\chi_{\alpha_1}''u+2\chi_{\alpha_1}'u')
\end{equation*}
the support of which is contained within the interval $[\alpha_1,2\alpha_1]$, therefore,
\begin{multline}\label{ineq2}
\norm{r^{\rfrac{3}{2}-t_0}[r^{\rfrac{n-1}{2}}(P-E\pm i\varepsilon){r^{-\rfrac{n-1}{2}}},\chi_{\alpha_1}]u}_{\Lp{2}{\R^n,drd\theta}} \leq h^2\norm{r^{\rfrac{3}{2}-t_0}u}_{\Lp{2}{\Ann{\alpha_1}{2\alpha_1}}} \\+ h\norm{r^{\rfrac{3}{2}-t_0}hu'}_{\Lp{2}{\Ann{\alpha_1}{2\alpha_1}}}.
\end{multline}
Using \eqref{VCond2}, we arrive at
\begin{multline}\label{ineq3}
\norm{\chi_{\alpha_1}r^{\rfrac{3}{2}-t_0}(V-E\pm i\varepsilon )u}_{\Lp{2}{\R_+\times\S^{n-1}}}\leq c_1a^{2-\delta}\norm{r^{-\rfrac{1}{2}-t_0}u}_{\Lp{2}{\Ann{0}{a}}}\\+\norm{r^{\rfrac{3}{2}-t_0}(V-E\pm i\varepsilon )u}_{\Lp{2}{\Ann{a}{2\alpha_1}}}.
\end{multline}
We substitute inequalities \eqref{ineq1} to \eqref{ineq3} into \eqref{ineq0} and subtract $\norm{r^{-\rfrac{1}{2}-t_0}u}_{\Lp{2}{\Ann{0}{a}}}$ to conclude the proof.
\end{proof}
\end{lemma}

\section{Constructing Phase and Weight Functions}

We return to the energy functional defined in \eqref{EnergyFunctional} and, for a weight function $w\in C^0[0,\infty)$ that is piecewise $C^1$, the distribution on $(wF)'$ on $(0,\infty)$,
\begin{multline*}
(wF)'= -2w\Re\ang{\Pconj u, u'}\mp2\varepsilon w\Im\ang{u,u'}+wq\ang{h^2r^{-2}\Lambda u,u} \\+ (4h^{-1}w\varphi '+w')\norm{hu'}^2 + (w(E+(\varphi ')^2-V))'\norm{u}^2 + 2w\Re\ang{h\varphi ''u,u'}.
\end{multline*}
where
\begin{equation}\label{defineq}
q:=2r^{-1} - \frac{w'}{w}.
\end{equation}
Using $2ab \geq -(\gamma a^2 + \gamma^{-1}b^2)$ for $\gamma > 0$, we get
\begin{align*}
(wF)'&\geq -\frac{\gamma_1w^2}{h^2w'}\norm{\Pconj u}^2\mp 2\varepsilon w\Im\ang{u,u'} \\&\quad+ wq\ang{h^2r^{-2}\Lambda u,u} + (4(1-\gamma_2^{-1})h^{-1}w\varphi '+ (1-\gamma_1^{-1}-\gamma_2^{-1})w')\norm{hu'}^2 \\&\quad+ (w(E+(\varphi')^2-V))'\norm{u}^2 - \frac{\gamma_2(w\varphi '')^2}{w'+4h^{-1}\varphi 'w}\norm{u}^2,
\end{align*}
for $\gamma_1,\gamma_2 > 0$. Setting $\eta > 0$ and letting $\gamma_1 = 2(1+\eta)\eta^{-1}$ and $\gamma_2 = 1+\eta$ we see that
\begin{multline*}
(wF)' \geq -\frac{2(1+\eta)w^2}{\eta h^2w'}\norm{\Pconj u}^2\mp 2\varepsilon w\Im\ang{u,u'} + wq\ang{h^2r^{-2}\Lambda u,u}\\+\left( 1-\frac{1}{1+\eta}\left(1+\frac{\eta}{2}\right)\right)w'\norm{hu'}^2+ (w(E+(\varphi')^2-V))'\norm{u}^2- \frac{(1+\eta)(w\varphi '')^2}{w'+4h^{-1}\varphi 'w}\norm{u}^2.
\end{multline*}
To control the term involving $\Lambda$ above, we will require that $q\geq 0$ and use the fact that $\Lambda \geq -\rfrac{1}{4}$ for any $n\in\N$ to get
\begin{multline}\label{wFBoundBelow}
(wF)' \geq -\frac{2(1+\eta)w^2}{\eta h^2w'}\norm{\Pconj u}^2\mp 2\varepsilon w\Im\ang{u,u'} + \left( 1-\frac{1}{1+\eta}\left(1+\frac{\eta}{2}\right)\right)w'\norm{hu'}^2 \\- \frac{h^2}{4r^2}wq\norm{u}^2+ (w(E+(\varphi')^2-V))'\norm{u}^2- \frac{(1+\eta)(w\varphi '')^2}{w'+4h^{-1}\varphi 'w}\norm{u}^2.
\end{multline}
Define
\begin{align*}
A(r):= (w(E+(\varphi ')^2-V))'-\frac{h^2wq}{4r^2},\qquad & B(r) := \frac{(w\varphi '')^2}{w' + 4h^{-1}\varphi 'w}.
\end{align*}
Let
\begin{equation*} 
b := \sup\left\lbrace r > 1 : V + \frac{1}{2}rV'\geq\frac{E}{4} \text{ and } V\geq\frac{E}{4}\right\rbrace,
\end{equation*}
which is finite because \eqref{VCond3} and \eqref{VCond6} imply $V, rV'\to 0$ as $r\to\infty$. Let
\begin{equation*}
M := 2\max\lbrace b, 6\sqrt{3}, 8KE^{-\rfrac{1}{2}}\rbrace, 
\end{equation*}
where the constant $K$ satisfies
\begin{equation}\label{constantC}
K\geq\max\left\lbrace \left(\frac{24c_1}{16-8\delta-(1+\eta)\delta^2}\right)^{\frac{1}{2}}, \sup_{1\leq r\leq b}\frac{1}{2}(1+r)^{\rfrac{3}{2}}\sqrt{1+y(r)+c_0m(r)}\right\rbrace.
\end{equation}
When $w',\varphi '\neq 0$, define
\begin{align}\label{ScriptWandPhi}
\mathcal{W} := \frac{w}{w'},\qquad& \Phi:=\frac{\varphi ''}{\varphi '}.
\end{align}

The goal is to construct weight and phase functions $w$ and $\varphi$ respectively, so that \eqref{wFBoundBelow} has a useful lower bound. This will be crucial in proving Theorem \ref{MainTheorem}.
Let
\begin{align}\label{defWeightPhase}
\varphi '(r) &= \left\lbrace\begin{array}{l r}
Kr^{-\rfrac{\delta}{2}} & 0\leq r\leq 1\\
\frac{2K}{1+r} & 1\leq r\leq \rfrac{M}{2}\\
\frac{8K}{M^2(1+\rfrac{M}{2})}(M-r)^2 & \rfrac{M}{2}\leq r< M\\
0 & r\geq M.
\end{array}
\right.,\\\nonumber \mathcal{W}(r) &= \left\lbrace\begin{array}{l r}
\frac{1}{2}r & 0<r< M\\
\min\left(\frac{Er^3}{4(c_0r^2m(r)+1)},\ang{r}^{2s}\right) & r > M
\end{array}\right. .
\end{align}
We can calculate the weight function $w$ to be
\begin{align}\label{wCalculated}
w(r) = \left\lbrace\begin{array}{l r}
K_1r^2 & 0\leq r< M\\
K_1M^2e^{\int_M^r\mathcal{W}^{-1}dr} & r\geq M
\end{array}
\right.
\end{align}
for some constant $K_1>0$.

\begin{lemma}\label{WeightAndPhase}
Fix $0<\eta<\ \frac{16-8\delta-\delta^2}{\delta^2}$ and suppose $V$ satisfies conditions \eqref{VCond1} through to \eqref{VCond6}. Then, for $h\in (0,h_0]$, where $h_0 < \frac{27E}{4(1+\eta)K}$ and for $\varphi$ and $w$ defined by \eqref{ScriptWandPhi} and \eqref{defWeightPhase},
\begin{equation*}
A - (1+\eta)B\geq\frac{E}{2}w'.
\end{equation*}
\begin{proof}
By slightly adapting (2.10) in \cite{LongRangeLipschitz}, we have the inequality
\begin{multline}\label{GaShInequality}
A-(1+\eta)B\geq w'\left[E + (\varphi ')^2\left(1+2\mathcal{W}\Phi-(1+\eta)\mathcal{W}\Phi^2\min \left(\mathcal{W}, \frac{h}{4\varphi '}\right)\right)\right.\\\left.-V-\mathcal{W}\left(V'+\frac{h^2q}{4r^2}\right)\right].
\end{multline}
For $0\leq r\leq 1$, we have $\varphi ' = Kr^{-\rfrac{\delta}{2}}$ and $\mathcal{W} = \frac{1}{2}r$, which implies $q=0$ and $\Phi = -\frac{\delta}{2r}$. Using \eqref{GaShInequality} in conjunction with \eqref{VCond2}, \eqref{VCond5} and \eqref{constantC}, we see that
\begin{equation*}
A-(1+\eta)B \geq w'\left[E + K^2r^{-\delta}\left(1-\frac{\delta}{2}-(1+\eta)\frac{\delta^2}{16}\right)-\frac{3}{2}c_1r^{-\delta}\right]\geq \frac{E}{2}w'.
\end{equation*}
For $1<r\leq \rfrac{M}{2}$, we have $\varphi ' = 2K(1+r)^{-1}$ and $\mathcal{W} = \frac{1}{2}r$, which implies $q=0$ and $\Phi = -\rfrac{1}{(1+r)}$. From the definition of $b$ and \eqref{VCond3} and \eqref{VCond6}, we have that
\begin{equation*}
V + \frac{1}{2}rV' \leq (y(r) + c_0m(r))\ind{1\leq r\leq b} + \frac{E}{4}\ind{r\geq b}.
\end{equation*}
Given that
\begin{equation*}
(\varphi ')^2\left(\mathcal{W}\Phi^2\min\left(\mathcal{W},\frac{h}{4\varphi '}\right)\right)\leq \frac{Khr}{4(1+r)^3}\leq \frac{Kh_0}{27},
\end{equation*}
the two inequalities above together with \eqref{GaShInequality} give us
\begin{align*}
A-(1+\eta)B &\geq w'\left[\frac{3E}{4}+\frac{4K^2}{(1+r)^3}-(1+\eta)\frac{Kh_0}{27}-(y(r)+c_0m(r))\ind{1\leq r\leq b}\right]\\&\geq w'\left[\frac{3E}{4}+\left(\frac{4K^2}{(1+r)^3}-(y(r)+c_0m(r))\right)\ind{1\leq r\leq b}-(1+\eta)\frac{Kh_0}{27}\right].
\end{align*}
Using \eqref{constantC} and the assumption that $h_0 < \frac{27E}{4(1+\eta)K}$, we see that,
\begin{equation*}
A-(1+\eta)B \geq \frac{E}{2}w',
\end{equation*}
for $1<r<\rfrac{M}{2}$.

For $\rfrac{M}{2}\leq r< M$, we have
\begin{gather*}
\varphi '(r) = \frac{8K}{M^2(1+\rfrac{M}{2})}(M-r)^2,\\
\Phi = \frac{-2}{M-r},\\
\mathcal{W} = \frac{1}{2}r,
\end{gather*}
and, because $\rfrac{M}{2}>b$, $V+\rfrac{1}{2}rV'\leq\rfrac{E}{4}$. We can then see that
\begin{gather*}
(\varphi ')^2(1+2\mathcal{W}\Phi)= (\varphi ')^2\left(1-\frac{2r}{M-r}\right)\geq -\frac{128K^2}{M^4\left(1+\rfrac{M}{2}\right)^2}r(M-r)^3\\\intertext{and}
(\varphi ')^2\mathcal{W}\Phi^2\min\left(\mathcal{W}, \frac{h}{4\varphi '}\right)\leq\frac{4Kh_0r}{M^2(1+\rfrac{M}{2})}.
\end{gather*}
Using \eqref{GaShInequality} and the fact that $\mathcal{W} = \frac{1}{2}r$ implies $q = 0$ and that $h_0 < \frac{27E}{4(1+\eta)K}$, we get
\begin{gather*}
A-(1+\eta)B\geq w'\left[\frac{3E}{4} - \frac{128K^2}{M^4(1+\rfrac{M}{2})^2}r(M-r)^3 - \frac{27Er}{M^2(1+\rfrac{M}{2})}\right],\\\intertext{which, together with}
r(M-r)^3\leq\frac{M^4}{16}\qquad\text{ and }\qquad r< M,\\\intertext{yields} A-(1+\eta)B\geq w'\left[\frac{3E}{4}-\frac{54E}{M^2}-\frac{32K^2}{M^2}\right].\\\intertext{We defined $M$ to be greater than or equal to $\max\lbrace 12\sqrt{3}, 16KE^{-\rfrac{1}{2}}\rbrace$, which gives us}
A-(1+\eta)B\geq \frac{E}{2}w',
\end{gather*}
for $\rfrac{M}{2}\leq r< M$.

On the region $r\geq M$, we have $\varphi ' = 0$, which reduces \eqref{GaShInequality} to
\begin{equation*}
A-(1+\eta)B\geq w'\left[E-V - \mathcal{W}V' - \frac{\mathcal{W}h^2q}{4r^2}\right].
\end{equation*} 
For $r\geq M$, $V\leq\rfrac{E}{4}$, $V'\leq c_0r^{-1}m(r)$ by \eqref{VCond6} and $q = 2r^{-1}-\mathcal{W}^{-1}$ by the definition of $q$. Therefore

\begin{equation*}
A-(1+\eta)B\geq w'\left[\frac{3E}{4}-\mathcal{W}\left(c_0r^{-1}m(r)+\frac{h^2}{2r^3}\right)\right].
\end{equation*}
By making the substitution $\mathcal{W} \leq \frac{Er^3}{4(c_0r^2m(r)+1)}$ we see that
\begin{equation*}
A-(1+\eta)B\geq \frac{E}{2}w'.
\end{equation*}
\end{proof}
\end{lemma}
\section{Carleman Estimates}

To be able to prove Theorem \ref{MainTheorem}, we first give a Carleman estimate. We begin by proving the following lemma.

\begin{lemma}\label{w'Estimate}
There are constants $C,h_0>0$ that are independent of $h$ and $\varepsilon$ so that
\begin{equation}
\int_{r,\theta}w'(|u|^2+|hu'|^2)drd\theta\leq \frac{C}{h^2}\int_{r,\theta}\ang{r}^{2s}|\Pconj u|^2drd\theta + \frac{C\varepsilon}{h}\int_{r,\theta}|u|^2drd\theta.
\end{equation}
for all $\varepsilon > 0$ and $h\in (0,h_0]$, and for all $u\in r^{\frac{n-1}{2}}\TestFn{\R^n}$.

\emph{The proof of this lemma follows a similar argument to that can be found in the proof of Lemma 3.2 in \cite{LongRangeLipschitz}, but is adapted for the use of a weight function $w$ that is quadratic near the origin.}

\begin{proof}
Starting with \eqref{wFBoundBelow} and applying Lemma \ref{WeightAndPhase}, for $h\in(0,h_0]$ we get
\begin{multline*}
(wF(r))'\geq -\frac{2(1+\eta)w^2}{\eta h^2w'}\norm{\Pconj u(r,\cdot)}^2\mp 2\varepsilon w\Im\ang{u(r,\cdot),u'(r,\cdot)} \\+ \left( 1-\frac{1}{1+\eta}\left(1+\frac{\eta}{2}\right)\right)w'\norm{hu'(r,\cdot)}^2+\frac{E}{2}w'\norm{u(r,\cdot)}^2,
\end{multline*}
where the norm and inner product used in this inequality are those of the space $\Lp{2}{\S^{n-1}_{\theta}}$, and $\eta > 0$ depends on $\delta$, as implied by Lemma \ref{WeightAndPhase}.

Integrating the inequality above with respect to $r$ from $0$ to $\infty$ and using the fact that $wF, (wF)'\in\Lp{1}{0,\infty}$ and $w(0) = 0$, we get $\int_0^\infty(wF)'dr = 0$. From \eqref{defWeightPhase} and \eqref{mCondition} we see that $\mathcal{W}^{-1}\in\Lp{1}{(0,\infty)}$, which implies the boundedness of $w$. This, together with the fact that $\frac{w}{w'}\lesssim \ang{r}^{2s}$, gives us
\begin{equation*}
\int_{r,\theta}w'(|u|^2+|hu'|^2)\lesssim\frac{1}{h^2}\int_{r,\theta}\ang{r}^{2s}|\Pconj u|^2+2\varepsilon\int_{r,\theta}w|uu'|.
\end{equation*}
We now use the Cauchy-Schwarz inequality on $2\varepsilon\int_{r,\theta}w|uu'|$, which gives
\begin{equation}\label{directlyabove}
\int_{r,\theta}w'(|u|^2+|hu'|^2)\lesssim\frac{1}{h^2}\int_{r,\theta}\ang{r}^{2s}|\Pconj u|^2+\frac{\varepsilon}{h}\int_{r,\theta}w|u|^2+\frac{\varepsilon}{h}\int_{r,\theta}w|hu'|^2.
\end{equation}
Let $\beta = \min\lbrace\frac{h}{c\varepsilon}, 1\rbrace$ where $c > 0$ is the implicit constant in \eqref{directlyabove}. Then for $0<r<\beta$, we have $\frac{c\varepsilon}{h}w \leq \frac{1}{2}w'$. The term $\int_{r\leq \beta}w|hu'|^2$ can be subtracted from both sides of \eqref{directlyabove}, giving
\begin{equation}\label{penultimate}
\int_{r,\theta}w'(|u|^2+|hu'|^2)\lesssim\frac{1}{h^2}\int_{r,\theta}\ang{r}^{2s}|\Pconj u|^2+\frac{\varepsilon}{h}\int_{r,\theta}|u|^2+\frac{\varepsilon}{h}\int_{\substack{r,\theta\\ r\geq \beta}}|hu'|^2,
\end{equation}
where we have used the fact that $w$ is bounded.

By letting $\chi(r)\in\TestFn{[0,\beta)}$ with $\chi\equiv 1$ on $[0,\rfrac{\beta}{2}]$ and $\psi = 1-\chi$, we use \eqref{GeneralPConj2} to get 
\begin{multline*}
\Re\int_{r,\theta}(\Pconj\psi u)\overline{\psi u}=\int_{r,\theta}|h(\psi u)'|^2 + \Re\int_{r,\theta}2h\varphi'(\psi u)'\overline{\psi u} + \int_{r,\theta}(h^2r^{-2}\Lambda\psi u)\overline{\psi u} \\ + \int_{r,\theta}h\varphi''|\psi u|^2+\int_{r,\theta}(V-E-(\varphi')^2)|\psi u|^2
\end{multline*}
and
\begin{equation*}
\int_{r,\theta}h\varphi''|\psi u|^2 = -\Re\int_{r,\theta}2h\varphi'(\psi u)'\overline{\psi u}.
\end{equation*}
Using the facts that on $\supp\psi$, $|V-E-(\varphi ')^2|$ and $r^{-2}$ are bounded and $\ang{r}^{-2s}\lesssim w'$ and that $\Lambda\geq -\rfrac{1}{4}$ together with the two equations above, we have that, for all $h\in(0,1]$ and $\gamma>0$,
\begin{equation}\label{hu'}
\int_{\substack{r,\theta \\ r\geq \beta}}|hu'|^2\lesssim\int_{r,\theta}|u|^2 + \frac{\gamma}{2}\int_{r,\theta}w'|u|^2+\frac{1}{\gamma}\int_{r,\theta}\ang{r}^{2s}|\psi\Pconj u|^2+\frac{h^2}{\gamma}\int_{\substack{r,\theta\\\frac{\beta}{2}\leq r\leq \beta}}|hu'|^2.
\end{equation}
After substituting \eqref{hu'} into \eqref{penultimate}, we choose $\gamma > 0$ small enough, and then making $h_0$ sufficiently small, giving
\begin{equation*}
\int_{r,\theta}w'(|u|^2+|hu'|^2)\lesssim \frac{1}{h^2}\int_{r,\theta}\ang{r}^{2s}|\Pconj u|^2 +\frac{\varepsilon}{h}\int_{r,\theta}|u|^2.
\end{equation*}
\end{proof}
\end{lemma}
We are now in a position to prove the Carleman estimate.

\begin{lemma}\label{Carleman}
There are constants $C_1, C_2 > 0$ independent of $h$ and $\varepsilon$ such that
\begin{gather}\label{1stCarleman}
\norm{\ang{r}^{-s}\ind{\leq M}v}_{\Lp{2}{\R^n}}^2\leq e^{\rfrac{C_1}{h}}\left(\norm{\ang{r}^{s}(P-E\pm i\varepsilon)v}_{\Lp{2}{\R^n}}^2 + \varepsilon\norm{v}_{\Lp{2}{\R^n}}^2\right),\\\label{2ndCarleman}
\norm{\ang{r}^{-s}\ind{\geq M}v}_{\Lp{2}{\R^n}}^2\leq \frac{C_2}{h^2}\norm{\ang{r}^{s}(P-E\pm i\varepsilon)v}_{\Lp{2}{\R^n}}^2 + \frac{C_2\varepsilon}{h}\norm{v}_{\Lp{2}{\R^n}}^2.
\end{gather}
for all $\varepsilon > 0$ and $h\in (0,h_0]$, and for all $v\in\TestFn{\R^n}$. 

\emph{
Here, the measure used to define the $L^2$ norms in use in \eqref{1stCarleman} and \eqref{2ndCarleman} is the Lebesgue measure on $\R^n$.
}

\begin{proof}
We begin with the proof by showing the inequality \eqref{1stCarleman} holds. We start by defining $u = r^{\frac{n-1}{2}}v$ where $v\in\TestFn{\R^n}$ then continue by separately considering $u$ in some small region near the origin and $u$ away from the origin, by writing
\begin{equation}\label{Split}
\int_{r,\theta}|\ang{r}^{-s}u|^2drd\theta = \int_{\substack{r,\theta\\0\leq r\leq a}}|\ang{r}^{-s}u|^2drd\theta + \int_{\substack{r,\theta\\r\geq a}}|\ang{r}^{-s}u|^2drd\theta.
\end{equation}
The second term, which describes $u$ away from the origin, can be estimated as follows 
\begin{equation}\label{1stTerm} 
\int_{\substack{r,\theta\\r\geq a}}|\ang{r}^{-s}u|^2\lesssim a^{-1}\int_{r,\theta}w'|u|^2
\end{equation}
for some constant $C>0$ that is independent of $h$ and $\varepsilon$. We have dropped the $drd\theta$ for ease of notation. As for the first term, we have the estimate
\begin{equation}\label{2ndTerm}
\int_{\substack{r,\theta\\ 0\leq r\leq a}}|\ang{r}^{-s}u|^2\lesssim a^{1+2t_0}\int_{\substack{r,\theta\\ 0\leq r\leq a}}|r^{-\rfrac{1}{2}-t_0}u|^2,
\end{equation}
where $t_0\in (-\rfrac{1}{2},0)$. Combining \eqref{1stTerm} and \eqref{2ndTerm} with \eqref{Split} gives
\begin{equation}\label{BothTerms}
\int_{r,\theta}|\ang{r}^{-s}u|^2 \lesssim a^{-1}\int_{r,\theta}w'|u|^2 + a^{1+2t_0}\int_{\substack{r,\theta\\ 0\leq r\leq a}}|r^{-\rfrac{1}{2}-t_0}u|^2.
\end{equation}
We now look towards Lemma \ref{lemNearOriginEstimate} to turn \eqref{BothTerms} into an estimate in terms of $(P-E\pm i\varepsilon)u$ and $u$, 
\begin{equation}\label{Pestimate}
\norm{r^{\rfrac{3}{2}-t_0}r^{\rfrac{n-1}{2}}(P-E\pm i\varepsilon){r^{-\rfrac{n-1}{2}}}u}_{\Lp{2}{\Ann{0}{2\alpha_1}}}^2\lesssim \int_{r,\theta}|\ang{r}^{s}\Pconj(e^{\rfrac{\varphi}{h}}u)|^2.
\end{equation}
Next, we make use of \eqref{VCond2} and that $w'\sim r$ on $\Ann{a}{2\alpha_1}$ to arrive at
\begin{equation}
\norm{r^{\rfrac{3}{2}-t_0}(V+E\pm i\varepsilon)u}_{\Lp{2}{\Ann{a}{2\alpha_1}}}^2\lesssim (1+a^{2-2t_0-2\delta})\int_{r,\theta}w'\left(|e^{\rfrac{\varphi}{h}}u|^2+|h(e^{\rfrac{\varphi}{h}}u)'|^2\right).
\end{equation}
Furthermore,since $t_0<0$, we have
\begin{align}\label{hu'estimate}
\norm{r^{\rfrac{3}{2}-t_0}u}_{\Lp{2}{\Ann{\alpha_1}{2\alpha_1}}}^2+\norm{r^{\rfrac{3}{2}-t_0}hu'}_{\Lp{2}{\Ann{\alpha_1}{2\alpha_1}}}^2&\lesssim \int_{r,\theta}w'\left(|e^{\rfrac{\varphi}{h}}u|^2+|h(e^{\rfrac{\varphi}{h}}u)'|^2\right).
\end{align}
From Lemma \ref{lemNearOriginEstimate} we have
\begin{multline}\label{AfterLemma2.2}
a^{1+2t_0}\int_{\substack{r,\theta\\ 0\leq r\leq a}}|r^{-\rfrac{1}{2}-t_0}u|^2\leq C\Upsilon (t_0)^2a^{1+2t_0}h^{-4}\left(\norm{r^{\rfrac{3}{2}-t_0}r^{\rfrac{n-1}{2}}(P-E\pm i\varepsilon){r^{-\rfrac{n-1}{2}}}u}_{\Lp{2}{\Ann{0}{2\alpha_1}}}^2\right. \\
\left. \quad+ \norm{r^{\rfrac{3}{2}-t_0}(V-E\pm i\varepsilon )u}_{\Lp{2}{\Ann{a}{2\alpha_1}}}^2\right.\\
+ \left. h\norm{r^{\rfrac{3}{2}-t_0}hu'}_{\Lp{2}{\Ann{\alpha_1}{2\alpha_1}}}^2 + h\norm{r^{\rfrac{3}{2}-t_0}hu}_{\Lp{2}{\Ann{\alpha_1}{2\alpha_1}}}^2\right).
\end{multline}
For small enough $h_0>0$, $a \sim h^{\frac{2}{2-\delta}}$ for $h\in (0,h_0]$, therefore $h^{-4}\sim a^{-2(2-\delta)}$, so $a^{1+2t_0}h^{-4}\sim a^{-3+2t_0-2\delta}$. Substituting inequalities \eqref{Pestimate} to \eqref{hu'estimate} into \eqref{AfterLemma2.2} allows us to obtain 
\begin{multline}\label{midway}
a^{1+2t_0}\int_{\substack{r,\theta\\ 0\leq r\leq a}}|r^{-\rfrac{1}{2}-t_0}u|^2\lesssim a^{-3+2t_0+2\delta}\int_{r,\theta}|\ang{r}^{s}\Pconj(e^{\rfrac{\varphi}{h}}u)|^2 \\+ (a^{-3+2t_0+2\delta}+a^{-1})\int_{r,\theta}w'\left(|e^{\rfrac{\varphi}{h}}u|^2+|h(e^{\rfrac{\varphi}{h}}u)'|^2\right).
\end{multline}
Because $t_0 > -\frac{1}{2}$ and $\delta\geq 0$, we get $a^{-3+2t_0+2\delta}\leq a^{-4}$ and therefore 
\begin{equation*}
a^{1+2t_0}\int_{\substack{r,\theta\\ 0\leq r\leq a}}|r^{-\rfrac{1}{2}-t_0}u|^2\lesssim a^{-4}\left(\int_{r,\theta}|\ang{r}^{s}\Pconj(e^{\rfrac{\varphi}{h}}u)|^2 +\int_{r,\theta}w'\left(|e^{\rfrac{\varphi}{h}}u|^2+|h(e^{\rfrac{\varphi}{h}}u)'|^2\right)\right).
\end{equation*}
From here, we make use of \eqref{BothTerms} followed by the substitution $a\sim h^{\frac{2}{2-\delta}}$ and arrive at
\begin{equation*}
\int_{r,\theta}|\ang{r}^{-s}u|^2 \lesssim h^{-\frac{8}{2-\delta}}\left(\int_{r,\theta}|\ang{r}^{s}\Pconj(e^{\rfrac{\varphi}{h}}u)|^2 +\int_{r,\theta}w'\left(|e^{\rfrac{\varphi}{h}}u|^2+|h(e^{\rfrac{\varphi}{h}}u)'|^2\right)\right).
\end{equation*} 
We now use Lemma \ref{w'Estimate} and a substitution $u = r^{\rfrac{n-1}{2}}v$ to arrive at \eqref{1stCarleman}.

To show \eqref{2ndCarleman}, we use the observation that $\ang{r}^{-2s}\lesssim w'$ for $r\geq M$, which implies
\begin{equation*}
\int_{\substack{r,\theta\\ r\geq M}}|\ang{r}^{-s}u|^2\lesssim\int_{r,\theta}w'(|u|^2+|hu'|^2),
\end{equation*}
then the use of Lemma \ref{w'Estimate} yields
\begin{equation*}
\int_{\substack{r,\theta\\ r\geq M}}|\ang{r}^{-s}u|^2 \lesssim\frac{1}{h^2}\int_{r,\theta}|\ang{r}^{s}\Pconj u|^2 + \frac{\varepsilon}{h}\int_{r,\theta}|u|^2
\end{equation*}
Making the substitution $u\mapsto e^{\frac{\varphi}{h}}r^{\frac{n-1}{2}}v$ gives
\begin{equation*}
\int_{\substack{r,\theta\\ r\geq M}}|\ang{r}^{-s}v|^2e^{\frac{2\varphi}{h}}r^{n-1} \lesssim e^{\frac{C_\varphi}{h}}\left(\frac{1}{h^2}\int_{r,\theta}|\ang{r}^{s}(P+E\pm i\varepsilon)v|^2r^{n-1} + \frac{\varepsilon}{h}\int_{r,\theta}|v|^2r^{n-1}\right)
\end{equation*}
where $C_\varphi := 2\max\varphi$. Furthermore, $2\varphi (r) = C_\varphi$ for $r\geq M$ because $\varphi ' \geq 0$ for all $r>0$ and $\varphi ' = 0$ for $r\geq M$. Dividing through by $e^{\frac{C_\varphi}{h}}$ gives us \eqref{2ndCarleman}.
\end{proof}
\end{lemma}

\section{Resolvent Estimates}
The goal of this section is to prove Theorem \ref{MainTheorem}. This proof follows the same argument as used in Section 5 of \cite{LongRangeLipschitz}.

\begin{proof}
Since increasing $s$ only decreases the weighted resolvent norms found in \eqref{GeneralFirstEstimate} and \eqref{GeneralSecondEstimate}, we can let $\rfrac{1}{2}<s<1$ without any loss of generality. Lemma \ref{Carleman} gives us $C_1,C_2,h_0>0$ such that
\begin{equation}\label{StartOfMainProof}
e^{-\rfrac{C_1}{h}}\norm{\ang{r}^{-s}\ind{\leq M}u}_{\Lp{2}{\R^n}}^2+\norm{\ang{r}^{-s}\ind{\geq M}u}_{\Lp{2}{\R^n}}^2\leq\frac{C_2}{h^2}\norm{\ang{r}^{s}(P-E\pm i\varepsilon)u}_{\Lp{2}{\R^n}}^2+\frac{C_2\varepsilon}{h}\norm{u}_{\Lp{2}{\R^n}}^2,
\end{equation}
for all $\varepsilon\geq 0$ and $h\in (0,h_0]$, and all $u\in\TestFn{\R^n}$. For any $\gamma, \gamma_0 > 0,$
\begin{align}\nonumber
2\varepsilon\norm{u}_{\Lp{2}{\R^n}}^2 &= -2\Im\ang{(P-E\pm i\varepsilon)u,u}_{{\Lp{2}{\R^n}}}\\\nonumber
&\leq\gamma^{-1}\norm{\ang{r}^s\ind{\leq M}(P-E\pm i\varepsilon)u}_{\Lp{2}{\R^n}}^2+\gamma\norm{\ang{r}^{-s}\ind{\leq M}u}_{\Lp{2}{\R^n}}^2\\\label{epsilon_u}&+\gamma_0^{-1}\norm{\ang{r}^s\ind{\geq M}(P-E\pm i\varepsilon)u}_{\Lp{2}{\R^n}}^2+\gamma_0\norm{\ang{r}^{-s}\ind{\geq M}u}_{\Lp{2}{\R^n}}^2.
\end{align}
We set $\gamma = h\rfrac{e^{-\rfrac{C_1}{h}}}{C_2}$ and $\gamma_0 = \rfrac{h}{C_2}$. Inequalities \eqref{StartOfMainProof} and \eqref{epsilon_u} imply, for some $C > 0$, all $\varepsilon\geq 0$, $h\in (0,h_0]$ and $u\in\TestFn{\R^n}$,
\begin{multline}\label{AlmostFinal}
e^{-\rfrac{C}{h}}\norm{\ang{r}^{-s}\ind{\leq M}u}_{\Lp{2}{\R^n}}^2+\norm{\ang{r}^{-s}\ind{\geq M}u}_{\Lp{2}{\R^n}}^2 \leq e^{\rfrac{C}{h}}\norm{\ang{r}^s\ind{\leq M}(P-E\pm i\varepsilon)u}_{\Lp{2}{\R^n}}^2\\+\frac{C}{h^2}\norm{\ang{r}^s\ind{\geq M}(P-E\pm i\varepsilon)u}_{\Lp{2}{\R^n}}^2.
\end{multline}
The final task is to use \eqref{AlmostFinal} to deduce
\begin{multline}\label{Final}
e^{-\rfrac{C}{h}}\norm{\ang{r}^{-s}\ind{\leq M}(P-E\pm i\varepsilon)^{-1}\ang{r}^{-s}f}_{\Lp{2}{\R^n}}^2+\norm{\ang{r}^{-s}\ind{\geq M}(P-E\pm i\varepsilon)^{-1}\ang{r}^{-s}f}_{\Lp{2}{\R^n}}^2\\\leq e^{\rfrac{C}{h}}\norm{\ind{\leq M}f}_{\Lp{2}{\R^n}}^2 + \frac{C}{h^2}\norm{\ind{\geq M}f}_{\Lp{2}{\R^n}}^2,
\end{multline}
for $\varepsilon> 0, h\in (0,h_0], f\in\Lp{2}{\R^n}$, from which Theorem \ref{MainTheorem} follows. We require a Sobolev space estimate followed by the application of a density argument that relies on \eqref{AlmostFinal}.

The operator
\[ [P,\ang{r}^s]\ang{r}^{-s} = (-h^2\Delta\ang{r}^s-2h^2(\nabla\ang{r}^s)\cdot\nabla)\ang{r}^{-s} \]
is bounded $H^2\to L^2$, so, for all $u\in H^2(\R^n)$ such that $\ang{r}^su\in H^2(\R^n)$,
\begin{align}\nonumber
\norm{\ang{r}^s(P-E\pm i\varepsilon)u}_{\Lp{2}{\R^n}}&\leq\norm{(P-E\pm i\varepsilon)\ang{r}^su}_{\Lp{2}{\R^n}} + \norm{[P,\ang{r}^s]\ang{r}^{-s}\ang{r}^su}_{\Lp{2}{\R^n}}\\\label{5.5}
&\leq C_{\varepsilon, h}\norm{\ang{r}^su}_{H^2(\R^n)},
\end{align}
for some constant $C_{\varepsilon, h}$ depending on $\varepsilon$ and $h$. Given $f\in\Lp{2}{\R^n}$, the function $u=\ang{r}^s(P-E\pm i\varepsilon)^{-1}\ang{r}^{-s}f\in H^2(\R^n)$ because $u=(P-E\pm i\varepsilon)^{-1}(f-w)$, where $w = \ang{r}^s[P,\ang{r}^{-s}]\ang{r}^s\ang{r}^{-s}u$ is $L^2$ because the operator $\ang{r}^s[P,\ang{r}^{-s}]\ang{r}^s$ is bounded $H^2\to L^2$ since $s<1$ and $\ang{r}^{-s}u=(P-E\pm i\varepsilon)^{-1}f$ is in $H^2$.

Now, choose a sequence $u_k\in\TestFn{\R^n}$ such that $u_k\to \ang{r}^s(P-E\pm i\varepsilon)^{-1}\ang{r}^{-s}f$ in $H^2(\R^n)$. Define $\tilde{u}_k:=\ang{r}^{-s}u_k$. Then, as $k\to\infty$,
\begin{equation}
\norm{\ang{r}^{-s}\tilde{u}_k-\ang{r}^{-s}(P-E\pm i\varepsilon)^{-1}\ang{r}^{-s}f}_{\Lp{2}{\R^n}}\leq\norm{u_k-\ang{r}^s(P-E\pm i\varepsilon)^{-1}\ang{r}^{-s}f}_{\Lp{2}{\R^n}}\to 0.
\end{equation}
Also, applying \eqref{5.5} gives
\begin{equation}
\norm{\ang{r}^s(P-E\pm i\varepsilon)\tilde{u}_k-f}_{\Lp{2}{\R^n}}\leq C_{\varepsilon, h}\norm{u_k-\ang{r}^s(P-E\pm i\varepsilon)^{-1}\ang{r}^{-s}f}_{H^2(\R^n)}\to 0.
\end{equation}
We then replace $u$ by $\tilde{u}_k$ in \eqref{AlmostFinal} and send $k\to\infty$ to attain \eqref{Final}.
\end{proof}




\nocite*{}
\bibliographystyle{alpha}
\bibliography{References}
\end{document}